\documentclass[12pt]{article}
\usepackage{amssymb, latexsym,pdfsync,amsmath,amsthm, ulem,hyperref,graphicx}
\usepackage{fullpage}
\usepackage{pgf,tikz}
\usepackage{mathrsfs}
\usetikzlibrary{arrows}
\newtheorem{theorem}{Theorem}
\newtheorem{corollary}{Corollary}

\newtheorem{proposition}{Proposition}

\theoremstyle{definition}
\newtheorem{definition}{Definition}
\newtheorem{remark}{Remark}
\newtheorem{example}{Example}
\newtheorem*{maintheorem}{Main Results}

\newcommand{\SSS}{\mathbb{S}}
\newcommand{\NN}{\mathbb{N}}

\newcommand{\FF}{\mathbb{F}}

\newcommand{\Fq}{\mathbb{F}_q}
\newcommand{\Fqn}{\mathbb{F}_{q^n}}

\newcommand{\Fqt}{\mathbb{F}_{q^2}}

\newcommand{\HH}{\mathcal H}

\newcommand{\G}{\mathcal G}
\newcommand{\C}{\mathcal C}
\newcommand{\cL}{\mathcal L}

\def\Fq{{\mathbb{F}}_q}
\def\Fqs{{\mathbb{F}}_{q^s}}
\def\Hom{\mathrm{Hom}}
\def\Aut{\mathrm{Aut}}

\def\GL{\mathrm{GL}}

\def\GammaL{\mathrm{\Gamma L}}
\def\dim{\mathrm{dim}}

\def\L{\mathcal{L}}
\def\id{\mathrm{id}}

\def\End{\mathrm{End}}

\def\gcrd{\mathrm{gcrd}}
\def\gcld{\mathrm{gcld}}
\def\Gal{\mathrm{Gal}}
\def\Ann{\mathrm{Ann}}
\def\I{\mathcal{I}}

\DeclareMathOperator{\modr}{~mod_r}

\newcommand{\npmatrix}[1]{\left( \begin{matrix} #1 \end{matrix} \right)}

\newtheorem{Proposition}{Proposition}

\newcommand{\rank}{\mathrm{rank}}

\begin{document}
\title{New Semifields and new MRD Codes from Skew Polynomial Rings}
\author{John Sheekey}
\date{}
\maketitle

\begin{abstract}
In this article we construct a new family of semifields, containing and extending two well-known families, namely Albert's generalised twisted fields and Petit's cyclic semifields (also known as Johnson-Jha semifields). The construction also gives examples of semifields with parameters for which no examples were previously known. In the case of semifields two dimensions over a nucleus and four-dimensional over their centre, the construction gives all possible examples.

Furthermore we embed these semifields in a new family of maximum rank-distance codes, encompassing most known current constructions, including the (twisted) Delsarte-Gabidulin codes, and containing new examples for most parameters.
\end{abstract}

\section{Introduction}

\subsection{History and Background}
A {\it semifield} is a division algebra $\SSS$ over a field with identity, in which multiplication is not assumed to be associative, and in which both left- and right-multiplication by a fixed element define an invertible endomorphism; in other words, for all $a,b\in \SSS$ there exist unique $x,y\in \SSS$ such that $ax=b$ and $ya=b$. Over the real numbers, a non-trivial example are the {\it octonions} discovered by Graves in 1843.  

If $\SSS$ has a finite number of elements, then it is a finite dimensional algebra over a finite field, and it suffices to define a semifield as a (nonassociative) division algebra. The study of finite semifields dates back to Dickson \cite{Dickson1906}, who constructed the first non-trivial examples in 1906. This contrasts with the famous Wedderburn-Dickson Theorem \cite{Dickson1905}, which states that any associative finite division algebra is in fact commutative, and hence a finite field.

The study of semifields was advanced in the mid 20th century, in particular by Albert, who constructed a family known as {\it generalised twisted fields} and developed the equivalence between semifields and certain classes of projective planes, and by Knuth, whose PhD thesis in 1965 \cite{Knuth1965} provided further constructions and details of the equivalence between semifields and nonsingular tensors. 

Ore \cite{Ore1932} and Jacobson \cite{Jacobson1934} each studied associative algebras arising from {\it skew-polynomial rings}, as generalisations of {\it cyclic algebras} and {\it Cayley-Dickson algebras}. Their results also implied a construction for semifields, though this was not explicitly stated. Petit \cite{Petit} made this construction explicit in 1965, though this work was not well-known to those studying finite semifields until recently; indeed the construction was rediscovered by Jha-Johnson in 1989 \cite{JHJO1989}, using the language of semilinear transformations. These are sometimes referred to as {\it cyclic semifields}. 

This construction was the largest known family of semifields at the time of writing, and thus gives the best lower bound for the number of semifields of a given order. Kantor \cite{Kantor2006} has conjectured that the number of semifields of a given order is not bounded above by any polynomial. This remains an open problem.

Semifields are studied not only in their own right, but also because of their connections to various other objects; for example classes of projective planes, nonsingular tensors, flocks of a quadratic cone, and ovoids. Very recently they have been studied due to the fact that they are special cases of {\it maximum rank distance codes}. We refer to the recent surveys \cite{LaPo2011}, \cite{Kantor2006}, \cite{Willems} for these connections, and for the known constructions. 

In this paper we use skew-polynomial rings to provide a new construction of semifields, and more generally of maximum rank-distance codes. This family contains both the generalised twisted fields and cyclic semifields as special cases, and hence is now the largest known construction. It also contains as a proper sub-family the maximum rank-distance codes recently introduced by the author in \cite{SheekeyMRD}, including the Delsarte-Gabidulin codes.

%


%
%

\subsection{Layout and Main Results}

In Section \ref{sec:known} we formally introduce semifields and maximum rank-distance codes, and illustrate how the former gives examples of the latter. We outline the necessary notions of equivalence, and recall the known constructions. In Section \ref{sec:skew} we  review skew-polynomial rings, and prove some properties which will be the foundations of the ensuing new construction. In Section \ref{sec:const} we use skew-polynomial rings to give a construction for a large class of MRD codes, including semifields. In Section \ref{sec:isnew} we show that this family contains new MRD codes and new semifields, including examples with nuclear parameters for which there were no previous examples known.

We summarise here our main results for finite fields (though they hold for a more general class of fields), and outline the main steps required for the proof. 
\begin{maintheorem}
\label{thm:mainstatement}
Let $\Fqn[x;\sigma]$ be the skew-polynomial ring, with $\sigma$ an automorphism of $\Fqn$ of order $n$. Let $F\in \Fq[y]$ be irreducible of degree $s$. Then the image of the set

\[
\{a_0+a_1x+\cdots+a_{sk-1}x^{sk-1}+\eta a_0^{\rho}x^{sk}:a_i \in \Fqn\} \subset \Fqn[x;\sigma]
\]
in the quotient ring 
\[
\frac{\Fqn[x;\sigma]}{(F(x^n))}\simeq M_n(\FF_{q^s})
\]
is a maximum rank-distance code, i.e. a set of size $q^{nks}$ such that the difference of any two elements has rank at least $n-k+1$, provided $\eta$ satisfies a certain condition on its norm.

Taking $k=1$ returns a family of semifields, including the generalised twisted fields, the cyclic semifields, and new semifields, including semifields with parameters for which there were no previously known examples. 

Taking $s=1$ returns the generalised twisted Gabidulin codes, and the family contains new MRD codes. 
\end{maintheorem}

In order to prove this we need to show the following.
\begin{itemize}
\item
Show that the above quotient ring is isomorphic to the above matrix ring (Section \ref{ssec:matfromskew});
\item
Show that the rank of the image of a skew-polynomial $f$ in the matrix ring is related to the greatest common right divisor of $f$ and $F(x^n)$ (Section \ref{ssec:rank});
\item
Find conditions for a skew-polynomial to be a divisor of $F(x^n)$, and thus low rank (Section \ref{ssec:div});
\item
Show that the above set avoids these divisors, and hence every element of the image has high rank, subsequently proving that we have constructed MRD codes/semifields (Section \ref{sec:const});
\item
Calculate the nuclear parameters of these sets, in order to prove inequivalence to known constructions (Section \ref{sec:isnew}).
\end{itemize}

The construction leads to a larger number of equivalence classes of both semifields and MRD codes than any known construction, and contains many previous constructions as proper subfamilies (Sections \ref{ssec:ns2}, \ref{sec:isnew}). Explicit representatives of new semifields are included in Section \ref{ssec:explicit}.

\subsection{Notation}

Throughout the rest of this paper, $K$ will denote a field, $L$ a cyclic Galois extension of degree $n$, and $\sigma$ a generator of $\Gal(L:K)$. We will consider rank-metric codes as subsets of $\End_K(L)$ (thus we restrict ourselves to the case of square matrices), and fix an isomorphism $\End_K(L)\simeq M_{n\times n}(K)$. When $K$ is finite, we will take $K=\Fq$, $L=\Fqn$, and $\sigma_q$ will denote the Frobenius automorphism $x\mapsto x^q$. Should $K$ not possess a cyclic Galois extension field of degree $n$, much of what follows is still valid if we replace $L$ by the vector space $K^n$.

\section{Semifields and Maximum Rank Distance Codes}\label{sec:RM}
\label{sec:known}

\subsection{Semifields}
\label{ssec:semifield}

Finite {\it presemifields} are division algebras with a finite number of elements in which multiplication is not assumed to be associative. If a presemifield contains a multiplicative identity, then it is called a {\it semifield}.

We will identify the elements of a semifield $n$-dimensional over $K$ with the elements of the field extension $L$. Two presemifields $\SSS=(L,\circ)$ and $\SSS'=(L,\circ')$ are said to be {\it isotopic} if there exist invertible additive maps $A,B,C$ from $L$ to $L$ such that
\[
(a\circ b)^A = a^B \circ' b^C
\]
for all $a,b\in L$. The set of semifields isotopic to $\SSS$ is called the {\it isotopy class} of $\SSS$, and is denoted by $[\SSS]$. All presemifields are isotopic to a semifield; for this reason, we will simply refer to semifields from here on.

There is a vast array of known constructions for semifields. We refer to \cite{LaPo2011} for further details.  A non-trivial example is Albert's {\it generalised twisted fields}, which are defined as follows. The elements are the elements of $L$, and multiplication is defined by
\[
a\circ b = ab-\eta a^{\sigma^i}b^{\sigma^j},
\]
where $\eta$ is a fixed element of $L$ such that $N(\eta)= \eta^{1+\sigma+\cdots+\sigma^{n-1}}\ne 1$.

For an element $x$ of a semifield $\SSS = (L,\circ)$, we can define the endomorphism of left multiplication $M_a\in \End_K(L)$ by
\[
M_a(b) := a\circ b.
\]
We then define the {\it spread set} of the semifield $\SSS$ to be the set
\[
\C(\SSS) := \{M_a:a \in \SSS\} \subseteq \End_K(L).
\]
As $\SSS$ is a (finite dimensional) division algebra, it is clear that every $M_a$ is invertible, except $M_0$. Furthermore as $M_{\lambda a + b}=\lambda M_a+M_b$ for all $\lambda\in K$, $a,b\in L$, it is clear that $\C(\SSS)$ forms a $K$-subspace of $\End_{K}(L)$. Hence we have the following well-known fact \cite{LaPo2011}.

\begin{proposition}
If $\SSS$ is a semifield $n$-dimensional over $K$, then $\C(\SSS)$ is an $n$-dimensional $K$-subspace of $\End_K(L)$ in which every non-zero element is invertible. Conversely, if $\C$ is such a subspace of $\End_K(L)$, then for any isomorphism $\phi$ from $L$ to $\C$ we can define a semifield $\SSS(\C) =(L,\circ_{\phi})$ by
\[
a\circ_{\phi} b := \phi(a)(b).
\]
\end{proposition}
Note that in some of the literature the spread set is defined using the endomorphisms of right multiplication; $R_a(b) = b\circ a$.  Note also that different choices for the isomorphism $\phi$ lead to isotopic semifields. A subspace of endomorphisms in which every non-zero element is invertible is a special case of a {\it maximum rank-distance code}, which we will introduce in the next section.

%

%
%
%
%
%
%

\subsection{Rank-metric codes}

A {\it rank-metric code} is a set $\C$ of vector-space homomorphisms from $K^m$ to $K^n$, $m\leq n$, equipped with the rank-distance function
\[
d(A,B) = \rank(A-B).
\]
In other words, $\C$ is a subset of $\Hom_K(K^m,K^n)$. If we choose a basis for each of $K^m$ and $K^n$, we may represent $\C$ by a set of $m \times n$ matrices $K$. Such a code is said to be {\it $K'$-linear} if it is closed under addition and $K'$-multiplication for some subfield $K'$ of $K$. In the case where $K$ possesses a field extension $L$ of degree $n$, we may also represent $\C$ by a set of vectors in $L^m$.


A subset $\C$ of $M_{m\times n}(K)$ must satisfy the {\it Singleton-like bound}; that is, if $\C$ has minimum distance $d$ (i.e. $\rank(A-B)\geq d$ for all $A,B\in \C$, $A\ne B$), then
\begin{itemize}
\item
if $\C$ is $K'$-linear, then $\dim_K'(\C)\leq n(m-d+1)[K:K']$;
\item 
if $K$ is finite, then $|\C|\leq |K|^{n(m-d+1)}$.
\end{itemize}

A code attaining this bound is said to be a {\it Maximum Rank Distance} code, or {\it MRD code}. Delsarte \cite{Delsarte1978} showed that MRD codes exist over every finite field for all parameters; in fact, he showed that $K$-linear MRD codes can be constructed for any finite field $K$ and any $m,n,d\leq \min\{m,n\}$. These were independently rediscovered by Gabidulin \cite{Gab1985} in the equivalent formulation as codes in $L^m$, and have come to be known as {\it Gabidulin codes} or {\it Delsarte-Gabidulin codes}.

Rank-metric codes have been studied in recent years in part due to their potential applications in {\it random network coding}; see for example \cite{SiKsKo2008}.

In the special case where $n=m=d$, $K$-linear MRD codes corresponds to semifields, while non-linear MRD codes correspond to algebraic structures called {\it quasifields}; see for example \cite{Handbook},\cite{Willems}. 

%
%
%

%
%

\subsection{Known Constructions for MRD codes}

Many of the known constructions exploit the correspondence between matrices (or vector space endomorphisms) and {\it linearized polynomials}; see for example \cite{SheekeyMRD}, \cite{Csaj}. The set of linearized polynomials is defined as follows.
\[
\cL(L,K,\sigma) := \{ f(X)= f_0X+f_1X^{\sigma}+\cdots+f_{n-1}X^{\sigma^{n-1}}:f_i \in L \}.
\]
Clearly any linearized polynomial defines a $K$-linear map of $L$, $[x\mapsto f(x)]\in \End_K(L)$. If we define the product of linearized polynomials to be composition, and identify $X^{\sigma^n}$ with $X$, the we get a ring isomorphism (see for example \cite{WuLiu}, \cite{Lidl}):
\[
\cL(L,K,\sigma)\simeq \End_K(L).
\]
Note that from this correspondence, we may say that a code is $L$-linear if the corresponding set of linearized polynomials is closed under multiplication by $L$. We can also move easily from linearized polynomials to vectors in $L^m$, simply by evaluating a polynomial $f$ on a set of elements $\{\alpha_1,\ldots,\alpha_m\}$ of $L$, linearly independent over $K$.
\[
f \mapsto (f(\alpha_1),\ldots,f(\alpha_m))\in L^m.
\]
Note also that if $K'$ is a subfield of $K$, and $K$ is a Galois extension of $K'$ with $\Aut(L:K')= \langle\rho\rangle$, then $\cL(L,K,\sigma)$ is naturally contained in $\cL(L,K',\rho)$, where $\sigma = \rho^{[K:K']}$.

We denote codes 
\[
\G_{k,\sigma} := \{ f_0X+f_1X^{\sigma}+\cdots+f_{k-1}X^{\sigma^{k-1}}:f_i \in \Fqn \}\subset \End_{K}(L).
\]
Delsarte \cite{Delsarte1978} showed that the set $\G_{k,\sigma_q} $ defines an MRD code for any $k$; that is, the set
\[
\G_{k,\sigma_q} := \{ f_0X+f_1X^{q}+\cdots+f_{k-1}X^{q^{k-1}}:f_i \in \Fqn \}\subset \End_{\Fq}(\Fqn).
\]
This can be easily seen: $\G_{k,\sigma_q}$ is clearly an $\Fq$-subspace of dimension $nk$, and as the number of zeroes of a polynomial in this set is at most $q^{k-1}$, then the rank of any element as an endomorphism of $\Fqn$ is at most $n-k+1$. Hence $\G_{k,\sigma_q}$ is clearly an $[n\times n,nk,n-k+1]$ MRD code. From this it is simple to construct $[m\times n,nk,n-k+1]$ MRD codes for any $m\leq n$. These are the {\it Gabidulin codes}. The codes $\G_{k,\sigma_q^s}$ were shown to be MRD in \cite{Roth1991}, and later independently in \cite{GaKs2005}, and are known as Generalised Gabidulin codes. The sets $\G_{k,\sigma}$ were shown to be MRD codes for any cyclic Galois extension $L$ in \cite{GoQu2009b}. Various properties of these codes have been studied in \cite{TrZh}.

In \cite{SheekeyMRD}, the following sets were introduced.
\[
\HH_{k,\sigma}(\eta,\rho) := \{ f_0X+f_1X^{\sigma}+\cdots+f_{k-1}X^{\sigma^{k-1}}+\eta f_0^{\rho}X^{\sigma^k}:f_i \in L \}\subset \End_K(L)
\]
Taking $L=\Fqn$, $X^{\sigma}=X^q$, $\eta=0$, returns the Gabidulin codes $\G_k$, while taking $L=\Fqn$, $X^{\sigma}=X^{q^s}$, $\eta=0$ returns the generalised Gabidulin codes of \cite{GaKs2005}.

In \cite{SheekeyMRD} it was shown that for $x^{\sigma}=x^q$, and $a^{\rho}=a^{\sigma^h}$, $\HH_k(\eta,\rho)$ is an MRD-code whenever $N(\eta)\ne (-1)^{nk}$. In \cite{LuTrZh2015} this was extended by allowing $\sigma$ to be any $\Fq$-automorphism of $\Fqn$, while in \cite{Ozbudak1} this was extended by allowing $\rho$ to be any automorphism of $\Fqn$, not necessarily fixing $\Fq$.

In the Section \ref{sec:const} we will generalise this further by looking instead at {\it skew polynomial rings}.

In \cite{TrZhHughes} a family of MRD codes was constructed in the case $n$ even by choosing $f_0=a,f_k=\eta b$ for $a,b\in \FF_{q^{n/2}}$, where $N(\eta)$ is a non-square in $\Fq$.

Other known constructions for MRD codes are: non-linear codes (\cite{CoMaPa}, \cite{DuSi}); rectangular matrices \cite{HoMa}; minimum distance $n-1$ \cite{Csaj}. In this paper we will be concerned only with linear codes of square matrices, and for general minimum distances. 

%

\subsection{Nuclei of semifields and equivalence of spread sets}
\label{ssec:semi}

The {\it left-, middle-, and right-nucleus} of a semifield are three subsets defined as follows.
\begin{align*}
{\NN}_l({\SSS})&:=\{a \in {L} ~|~ a \circ (b\circ c)=(a\circ b)\circ c, ~\forall b,c \in {L}\}, \\
{\NN}_m({\SSS})&:=\{b \in {L} ~|~  a \circ (b\circ c)=(a\circ b)\circ c, ~\forall a,c \in {L}\}, \\
{\NN}_r({\SSS})&:=\{c \in {L} ~|~  a \circ (b\circ c)=(a\circ b)\circ c, ~\forall b,c \in {L}\}.
\end{align*}
The {\it nucleus} $\NN(\SSS)$ of $\SSS$ is the intersection of these three sets, and the {\it centre} $Z(\SSS)$ is defined as
\[
Z(\SSS) := \{a \in \NN(\SSS)~|~a\circ b = b \circ a ~\forall b\in L\}.
\]
Each of the nuclei are division rings, and the centre is a field. When $\SSS$ is finite, the nuclei are also fields. The centre is the largest field over which $\SSS$ is an algebra. Note that this definition does not extend to presemifields.

From a semifield $\SSS=(L,\circ)$ we have seen that we can obtain an MRD code in $\End_K(L)$. In fact we can obtain an MRD code in a smaller space of endomorphism, as follows. We can view $\SSS$ as a right-vector space over its right nucleus, and each $M_a$ is an $\NN_r(\SSS)$-endomorphism of $\SSS$. Then we define the set $\C(\SSS)$ as follows.
\[
\C(\SSS) := \{M_a:a \in \SSS\}\subseteq \End_{\NN_r(\SSS)}(\SSS)\subseteq \End_{Z(\SSS)}(\SSS).
\] 
This is also referred to as the {\it spread set} of the semifield $\SSS$. Similar to Section \ref{ssec:semifield}, we have the following well-known fact \cite{LaPo2011}.

\begin{proposition}
If $\SSS$ is a semifield, then $\C(\SSS)\subset \End_{\NN_r(\SSS)}(L)$ is a $Z(\SSS)$-linear MRD code with minimum distance $\dim_{\NN_r(\SSS)}(L)$.
\end{proposition}

Furthermore, isotopy of semifields translates into an equivalence of spread sets as follows (paraphrased from \cite[Theorem 7]{LaPo2011}):
\begin{proposition}
Two semifields $\SSS$ and $\SSS'$ with $\NN_r(\SSS) = \NN_r(\SSS')$ are isotopic if and only if there exist $A,B\in \End_{\NN_r(\SSS)}(L)$ invertible, and $\rho \in \Aut(\NN_r(\SSS))$, such that
\[
\C(\SSS) = \{A X^{\rho}B:X\in \C(\SSS')\}.
\]
%
%
\end{proposition}
Here we take $X^{\rho}(v) := (X(v^{\rho^{-1}}))^{\rho}$ for $v\in L$. We will relate the nuclei to sets of endomorphisms in the next subsection.



\subsection{Equivalence, Automorphism Groups, and Idealisers of rank-metric codes}

There are differing definitions of equivalence of rank-metric codes in the literature. We define an action of elements of $\Gamma:= \GL(L,K)\times \GL(L,K)\times \Aut(K)$ on $\End_K(L)$ as follows.
\[
(A,B,\rho):X \mapsto AX^{\rho}B.
\]
If we choose a basis and take a matrix representing $X$, then $\rho$ is applied entry-wise. Then $\Gamma$ forms a group with the product
\[
(A,B,\rho)(C,D,\tau) = (AC^{\rho},D^{\rho}B,\rho\tau).
\]
We denote the subgroup $\GL(L,K)\times \GL(L,K)\times \langle \id\rangle$ by $G$. The group $\Gamma$ is isomorphic to the subgroup of $\GammaL(\End_K(L),K)$ fixing the set of rank one elements; in the projective space this set is referred to as the {\it Segre variety}, and the induced group is the setwise stabiliser of the Segre variety in the collineation group of the projective space. The group $G$ is isomorphic to the subgroup of $\GL(\End_K(L),K)$ fixing the set of rank one elements. We say that two subsets of $\End_K(L)$ are {\it equivalent} if they lie in the same $\Gamma$-orbit, and {\it projectively equivalent} if they lie in the same $G$-orbit.

\begin{definition}
For a code $\C$, its {\it automorphism group} $\Aut(\C)$ is defined as
\[
\Aut(\C):= \{(A,B,\rho):A,B \in \GL(L,K), \rho \in \Aut(K) , A\C^{\rho} B = \C\}.
\]
We define the {\it projective automorphism group} $\Aut_K(\C)$ of $\C$ as
\[
\Aut_K(\C):= \{(A,B):A,B \in \GL(L,K), A\C B = \C\}.
\]
\end{definition}
The automorphism groups of twisted gabidulin codes was calculated in \cite{SheekeyMRD} for $s=0$ and \cite{LiebNebe}, \cite{TrZh} for $s>0$. When $\C=\C(\SSS)$ is the spread set of a semifield, then the automorphism group coincides with the {\it autotopy group} of the semifield.

%
%

In \cite{MarPol}, the nuclei of a semifield were defined in terms of sets of endomorphisms. In \cite{TrZh} and \cite{LiebNebe}, these were extended to all codes. Differing terminology was used in each; here we use slightly different definitions which are more convenient for this paper. 
\begin{definition}
The {\it left idealiser} $\I_l(\C)$ is defined as
\[
\I_{\ell}(\C) = \{A:A \in \End_K(L), A\C\subseteq \C\}
\]
The {\it right idealiser} $\I_r(\C)$ is defined as
\[
\I_r(\C) = \{A:A \in \End_K(L), \C A\subseteq\C\}
\]
The {\it centraliser} $C(\C)$ is defined as
\[
C(\C) = \{A:A \in \End_K(L), AX=XA ~\forall~X\in \C\}.
\]
The {\it centre} $Z(\C)$ of $\C$ is defined as the intersection of the left idealiser and the centraliser.
\[
Z(\C) = \I_{\ell}(\C)\cap C(\C).
\]
\end{definition}

We now show that these objects can be seen as invariants of codes. For the centre and centraliser to be invariants, we need to assume that the identity is contained in each code. However for most codes of interest, for example MRD codes, this is not a major restriction.
\begin{proposition}
\label{prop:idequiv}
Suppose $\C$ and $\C'$ are two equivalent codes. Then there exist invertible endomorphisms $A,B$, and $\rho\in \Aut(K)$, such that 

\begin{align*}
\I_{\ell}(\C) &= (A^{-1}\I_{\ell}(\C')A)^{\rho},\\
\I_r(\C) &= (B\I_{\ell}(\C')B^{-1})^{\rho}.
\end{align*}
Furthermore, if both $\C$ and $\C'$ contain the identity, then 
\begin{align*}
C(\C) &= (A^{-1}C(\C')A)^{\rho},\\
Z(\C) &= (A^{-1}Z(\C')A)^{\rho}.
\end{align*}
If $L$ is a finite field, and if both $\C$ and $\C'$ contain the identity, then
\[
|\I_{\ell}(\C)| = |\I_{\ell}(\C')|;\quad |\I_{r}(\C)| = |\I_{r}(\C')|;\quad |C(\C)| = |C(\C')|;\quad |Z(\C)| = |Z(\C')|.
\]
\end{proposition}

\begin{proof}
Let $\C' = A\C^{\rho^{-1}}B$ for invertible $A,B$, and $\rho\in \Aut(K)$. Then 
\[
D \in \I_{\ell}(\C')\Leftrightarrow D\C' = \C' \Leftrightarrow DA\C^{\rho^{-1}}B = A\C^{\rho^{-1}}B\Leftrightarrow (A^{-\rho}D^{\rho}A^{\rho})\C = \C \Leftrightarrow A^{-\rho}D^{\rho}A^{\rho}\in \I_{\ell}(\C).
\]
The proof for the right idealiser is similar.

Now suppose $D\in C(\C')$. Then 
\[
DAX^{\rho^{-1}}B = AX^{\rho^{-1}}BD
\]
for all $X\in \C$. Hence
\[
(A^{-\rho}D^{\rho}A^{\rho})X = X(B^{\rho}D^{\rho}B^{-\rho})
\]
for all $X\in \C$. Since the identity is in $\C$, we have $A^{-\rho}D^{\rho}A^{\rho}=B^{\rho}D^{\rho}B^{-\rho}$, and so $A^{-\rho}D^{\rho}A^{\rho}\in C(\C)$, implying that $Z(\C) \supseteq (A^{-1}Z(\C')A)^{\rho}$. Repeating this argument interchanging $\C$ and $\C'$ completes the proof.
\end{proof}

These sets are a useful generalisation of the nuclei and centre of a semifield, as we now demonstrate. 

For a semifield $\SSS$ with left-, middle-, and right-nuclei $\NN_l,\NN_m$, and $\NN_r$ respectively, and centre $Z$, we note the following isomorphisms with sets of endomorphisms, which by abuse of terminology we will also refer to as the left-nucleus etc.
\begin{align*}
\NN_l &\simeq \{M_a:a \in \NN_l\} \subseteq \End_K(L)\\
\NN_m &\simeq \{M_a:a \in \NN_m\} \subseteq \End_K(L)\\
\NN_r &\simeq \{R_a:a \in \NN_r\} \subseteq \End_K(L)\\
Z &\simeq \{R_a:a \in Z\} = \{M_a:a \in Z\} \subseteq \End_K(L)\\
\end{align*}
As $\SSS$ is a semifield, we have that that the identity and zero endomorphisms are in each of these sets. In the next proposition we relate the nuclei to the idealisers and centraliser; much of this is taken from \cite{MarPol}, and translated into the terminology of this paper. However here we use a different characterisation of the right nucleus, due to the fact that the notion of the {\it dual} of a semifield does not generalise to all codes.

\begin{proposition} \label{prop:nuc} Let $\SSS$ be a semifield, with nuclei as above, and $\C=\C(\SSS)$ the spread set of $\SSS$.
\begin{itemize}
\item[(i)]The left nucleus $\NN_l$ is isomorphic to  $\I_l(\C)$.
\item[(ii)]The middle nucleus $\NN_m$ is isomorphic to  $\I_r(\C)$.
\item[(iii)]The right nucleus $\NN_r$ is isomorphic to $C(\C)$ .
\item[(iv)] The centre $Z$ is isomorphic to $Z(\C)$.
\end{itemize}
\end{proposition}

\begin{proof} 
(i) and (ii) are adapted from \cite{MarPol}. Since the identity is contained in $\C(\SSS)$, we have that $\I_l(\SSS)\subseteq \C(\SSS)$. It holds that $\I_l(\SSS)=\{M_a:a \in \NN_l\}$ since $a\in \NN_l$ if and only if $M_aM_b = M_{a\circ b}$ for any $b\in \SSS$. Similarly we have that $\I_r(\SSS)=\{M_a:a \in \NN_m\}$.
%
%
%

(iii) We show that $\{R_a:a \in \NN_r(\SSS)\}$ is equal to the centraliser of $\C$. First note that $R_aM_b(x) = (b\circ x)\circ a = b\circ(x\circ a) = M_bR_a(x)$ for all $x$, and so $R_aM_b=M_bR_a$ for all $b$, implying that $R_a$ is in the centraliser of $\C$.

Now suppose $AM_b=M_bA$ for all $b\in \SSS$. Since $\SSS$ is a semifield, it has an identity element $e$. Let $a = A(e)$. Then $R_a(b) = b\circ a = M_b(a) = M_bA(e)=AM_b(e) = A(b\circ e) = A(b)$ for all $b$, and hence $A = R_a$.

(iv) If $a\in Z(\SSS)$, then $M_a=R_a$, and $M_aM_b=M_bM_a$ for all $b\in \SSS$, and so $M_a\in  \I_{\ell}(\C)\cap C(\C)$. Conversely, if $A\in \I_{\ell}(\C)\cap C(\C)$, then $A = M_a$ for some $a\in \NN_l$, and $M_aM_b=M_bM_a$ for all $b\in \SSS$. Therefore $a\circ b = M_aM_b(e) = M_bM_a(e) = b\circ a$ for all $b\in \SSS$, and so $M_a=R_a$, implying $a\in \NN_r$. Furthermore, $M_a\in \I_r(\SSS)$, and so $M_a=R_c$ for some $c\in \SSS$. Thus $a= M_a(e) = R_c(e)= c$, and so $a\in \NN_m$. Therefore $a$ is in the intersection of the three nuclei, and commutes with every element of $\SSS$, and hence $a\in Z(\SSS)$. Therefore $\I_{\ell}(\C)\cap C(\C) = \{M_a:a\in Z(\SSS)\}$, completing the proof.
\end{proof}

\begin{remark}
Note that in the above proposition we require $\SSS$ to be a semifield, rather than a presemifield. If this assumption is removed, then (iii) is no longer necessarily true. This is the same reason that we restrict to codes containing the identity in Proposition \ref{prop:idequiv}. 
%
%
\end{remark}

%

\begin{definition}
The {\it nuclear parameters} of a finite semifield $\SSS$ are defined as the tuple
\[
(|\SSS|,|\NN_\ell|,|\NN_m|,|\NN_r|,|Z|).
\]
By the previous discussion, we may extend this definition to all codes containing the identity, by defining the nuclear parameters of a code $\C$ by 
\[
(|\C|,|\I_{\ell}(\C)|,|\I_{\ell}(\C)|,|C(\C)|,|Z(\C)|).
\]
If a code $\C$ contains an invertible element but does not contain the identity, we define its nuclear parameters to be the nuclear parameters of any code $\C'$ which contains the identity and is equivalent to $\C$.

For codes over infinite fields, if $Z(\C)$ is a field then, setting $K=Z(\C)$, we can also define the nuclear parameters as follows:
\[
(\dim_K(\C),\dim_K(\I_\ell(\C)),\dim_K(\I_r(\C)),\dim_K(C(\C))).
\]
\end{definition}

\begin{remark}
Each of the idealisers, centraliser, and centre of $\C$ form a subring of $\End_K(L)$. If $\C$ is a semifield spread set then each form a division ring, and if $\C$ is a {\it finite} semifield spread set then each form a field.

If $\C$ contains the identity, then $Z$ is a commutative subring. If $\C$ is a semifield spread set then $Z$ is a field.

Note that if $\C$ is a commutative subring of $\End_K(L)$, then all of these sets coincide. 
\end{remark}

It is easy to see that idealisers and centraliser are related to subgroups of the automorphism group;
\[
\I_l(\C) \supseteq \{A : (A,I)\in \Aut(\C)\}\cup \{(0,I)\};
\]
\[
\I_r(\C) \supseteq \{A : (I,A)\in \Aut(\C)\}\cup \{(I,0)\};
\]
\[
C(\C) \supseteq \{A : (A,A^{-1})\in \Aut(\C)\}\cup\{(0,0)\}.
\]
In the case that $\C$ contains the identity and such that all non-zero elements of $\C$ are invertible, these become equalities.

\subsection{Nuclear parameters of known constructions for semifields and MRD codes}

Here we collect information on the nuclear parameters of the relevant known constructions. We omit constructions which only occur in certain characteristics. We also omit construction which are two-dimensional over a nucleus, as there are a vast array of such constructions.

The nuclei of Generalised Twisted Fields were calculated by Albert \cite{Albert1961}: these are semifields with multiplication
\[
a\circ b = ab-\eta a^{p^i}b^{p^j},\quad x,y\in \FF_{p^{ne}},
\]
which have nuclear parameters
\[
(p^{ne},p^{(ne,i)},p^{(ne,j-i)},p^{(ne,j)},p^{(ne,i,j)}).
\]

The Pott-Zhou commutative semifields \cite{PottZhou}, which are defined for integers $n,i,j$ such that $\frac{n}{(n,i)}$ is odd and $j\leq n$, have nuclear parameters
\[
(p^{2n},p^{(n,i,j)},p^{(n,i)},p^{(n,i,j)},p^{(n,i,j)}).
\]

The Petit (or cyclic) semifields (see \cite{Petit}, \cite{LaShSkew}), which we will consider in the next section, have nuclear parameters
\[
(q^{ns},q^{n},q^n,q^s,q),
\]
where $q$ is a power of the prime $p$.

Cyclic semifields were generalised in \cite{JOMAPOTR2009}, but specific to the case of two-dimensional over a nucleus.


\section{Skew Polynomial Rings}\label{sec:skew}
In this section we recall basic properties of skew polynomial rings, review the construction of semifields due to Petit, and prove some important facts about skew polynomials which will allow us to extend this construction.

\label{sec:skew}

The {\it skew polynomial ring}  $L[x;\sigma,\delta]$ is a ring where
\begin{itemize}
\item
$\delta$ is a {\it $\sigma$-derivation} of $L$; that is, an additive map satisfying $(ab)^{\delta} = a^\sigma b^\delta+a^\delta b$;
\item
the elements of $L[x;\sigma,\delta]$ are polynomials in the indeterminate $x$; addition is polynomial addition;
\item
multiplication is $K$-bilinear, and satisfies $xa = a^{\sigma}x+a^{\delta}$ for all $a\in L$.
\end{itemize}

In this article we will restrict ourselves to rings $L[x;\sigma,0]=:L[x;\sigma]$, i.e. rings with trivial derivation. Such rings are sometimes called {\it twisted polynomial rings}. In the case where $L$ is finite, this is no loss of generality, as every skew polynomial ring is isomorphic to a twisted polynomial ring (see e.g. \cite{JacobsonBook}, \cite{SheekeyThesis}).

The study of these rings dates back to Ore \cite{Ore1933}, who showed the following properties.
\begin{theorem}
\label{thm:skewprop}
Let $R=L[x;\sigma]$, where $\sigma$ is not the identity automorphism. Then
\begin{enumerate}
\item[(i)]
$R$ is a non-commutative integral domain;
\item[(ii)]
$R$ is {\it not} a unique factorisation domain;
\item[(iii)]
$R$ is both {\it left- and right-Euclidean domain}, with the usual degree function;
\item[(iv)]
the centre of $R$ is $K[x^n;\sigma]\simeq K[y]$, where $K$ is the fixed field of $\sigma$ and $n$ is the order of $\sigma$;
\item[(v)]
$R$ is a {\it Principal Ideal Domain}, with prime two-sided ideals generated by elements of the form $F(x^n)$, where $n$ is the order of $\sigma$, and $F$ is an irreducible element of $K[y]$;
\item[(vi)]
if $f$ has factorisations $f=g_1\cdots g_k=h_1\cdots h_m$, then $k=m$ and there is a permutation $\pi\in S_k$ such that $\deg(g_{\pi(i)}) = \deg(h_i)$ for all $i$;
\item[(vii)]
if $K$ is finite and $F$ is irreducible in $K[y]$, then every irreducible right divisor of $F(x^n)$ in $R$ has degree equal to $\deg(F)$, and every right divisor of $F(x^n)$ in $R$ has degree equal to $k\deg(F)$ for some $k\in \{1,\ldots,n\}$.
\end{enumerate}
\end{theorem}
Further background and facts about skew polynomial rings can be found in \cite{Giesbrecht1998}, \cite{SheekeyThesis}, \cite{LeBorgne}.

\begin{definition}
The {\it minimal central left multiple} of an element $f\in R$ is the unique monic polynomial $F(x^n)$ of minimal degree in $Z(R)$ such that $F(x^n) = gf$ for some $g\in R$. 
\end{definition}

From the above properties of skew polynomial rings, we have the following fact.
\begin{corollary}
If $f$ is an irreducible element of $R$ of degree $s$, then its minimal central left multiple $F(x^n)$ is such that $F$ is an irreducible element of $K[y]$ of degree $s$.
\end{corollary}

\subsection{Petit's (cyclic) semifields}

Semifields have be constructed from skew polynomial rings in an analogous way to the construction of a field extension from a polynomial ring. We now briefly recall this construction. For details we refer to \cite{Petit}, \cite{LaShSkew}, \cite{SheekeyThesis}, \cite{Dempwolff2011}, \cite{Pumpluen}.

Let $f$ be an irreducible element of $R$ of degree $s$. Define $V$ as the set of elements of $R$ of degree at most $s-1$. Define a semifield $\SSS_f$ whose elements are the elements of $V$, with  multiplication defined by
\[
a \circ b := ab \modr f,
\]
where $\modr f$ denotes the remainder on {\it right} division by $f$; this is well defined, and non-zero for $a,b\ne 0$, due to the fact that $R$ is a right-Euclidean domain. 

In \cite{Petit} this construction was given, and the nuclei were calculated. Alternative proofs can be found in \cite{LaShSkew}, \cite{Dempwolff2011}. For $L=\Fqn$ a finite field, it holds that
\[
|\SSS_f|=q^{ns};\quad |\NN_l| =|\NN_m|= q^n;\quad |\NN_r| = q^s;\quad |Z| = q.
\]

This extended a construction of Sandler \cite{Sandler1962}, which in turn can be seen as a nonassociative generalisation of Cayley-Dickson algebras and cyclic algebras. In \cite{LaShSkew} it was shown that the {\it cyclic} or {\it Jha-Johnson semifields} of \cite{JHJO1989} are in fact isotopic to Petit's semifields. It was also shown that this construction contains some families of semifields defined by Knuth \cite{Knuth1965}.  

In \cite{LaShSkew} and \cite{Dempwolff2011} the isotopy classification for Petit's semifields was considered. The following upper bound for the number of isotopy classes was proved in \cite{LaShSkew}, improving on that in \cite{KALI2008}. If $\mathcal{P}_{n,s}$ denotes the set of isotopy classes from Petit's construction with $L=\Fqn$, $\deg(f)=s$, then 
\[
|\mathcal{P}_{n,s}| \leq q^s/s
\] 
In particular it was shown that if two irreducible polynomials $f,g$ have the same {\it minimal central left multiple}, then $\SSS_f$ and $\SSS_g$ are isotopic.
 
We will now reinterpret this construction in a way that will allow us to extend this construction to a family of MRD codes for all minimum distances, and also to extend the construction to obtain new semifield and further new MRD codes, in an analogous way to the construction of the twisted Gabidulin codes.

\subsection{Matrices from skew polynomial rings}
\label{ssec:matfromskew}

Consider the skew polynomial ring $R = L[x;\sigma]$. A polynomial defines a two-sided ideal in $R$ if and only if it is of the form $F(x^n)$ for some $F(y)\in K[y]$, and the ideal $RF(x^n)$ is maximal (as a two-sided ideal) if and only if $F(y)$ is irreducible in $K[y]$. All left-ideals are of the form $Rf$ for some $f\in R$, and a left-ideal is maximal if and only if $f$ is irreducible. Furthermore, the unique maximal two-sided ideal containing the maximal left-ideal $Rf$ is precisely $RF(x^n)$, where $F$ is the minimal central left multiple of $f$.

Suppose from now on that $F$ is a monic irreducible polynomial in $K[y]$, and $\deg(F) = s$. Then $R/RF(x^n)$ is a simple algebra, with centre $Z(R)/RF(x^n)\simeq K[y]/\langle F(y)\rangle$; this is a field extension of $K$ of order $s$, which we will denote by $E_F$. Then we have the following (which can be found for example in \cite{Giesbrecht1998}).
\begin{Proposition} For any irreducible polynomial $F$ in $K[y]$ with $\deg(F) = s$, we have
\[
R_F := \frac{R}{RF(x^n)} \simeq \End_{E_F}(E_F^n)\simeq M_n(E_F).
\]
\end{Proposition}

When $F(y)=y-1$, we have $E_F=K$ and we get the previous correspondence between linearized polynomials and $M_n(K)$, if we identify $x^i$ with $X^{\sigma}$, and skew polynomial multiplication with composition. The correspondence is then $M_n(K)\simeq \frac{L[x;\sigma]}{(x^n-1)}$.

%

\subsection{Formulation as maps on a vector space}

We can see the action of $R_F \simeq M_n(E_F)$ as a linear map on a vector space isomorphic to $E_F^n$ as follows. 

Let $f$ be a monic irreducible divisor of $F(x^n)$. By Theorem \ref{thm:skewprop} we must have $\deg(f)=s$, and $F(x^n)$ is the minimal central left multiple of $f$.

Let $V_f = \{a+Rf:a \in R\}$ be the set of cosets of the maximal left-ideal defined by $f$. Let $E_f = \{z+Rf:z \in Z(R)\}$. Note that $E_F$ is isomorphic to $E_f$, since for any $z\in Z(R)$, we have that $z\in Rf$ if and only if $z\in RF(x^n)$, by definition of minimal central left multiple..

For $z\in Z(R)$, $a\in R$, we define $(a+Rf)(z+Rf) = az+Rf$. Note that this is well-defined only because $z$ is in the centre of $R$. Under this multiplication, $E_f$ is a finite field isomorphic to $\Fqs$. Furthermore, $V_f$ is a vector space over $E_f$, of dimension $n$, and so we can identify it with $E_f^n$.

We can then view the elements of $\frac{R}{RF(x^n)}$ as $E_f$-endomorphisms of $V_f$ by identifying an element $a+RF(x^n)$ with the map
\[
M_a:b+Rf \mapsto ab+Rf.
\]
This is well-defined; if $a+RF(x^n)=a+RF(x^n)$, then $a' = a+cF(x^n)$. Therefore $M_{a'}(b+Rf) = a'b+Rf = (ab+cF(x^n)b)+Rf = (ab+cbF(x^n))+Rf = ab+Rf = M_a(b+Rf)$ for all $b\in R$, since $F(x^n)\in Z(R)\cap Rf$.

It is not straightforward to find a canonical $E_f$-basis for $V_f$. However we will consider some particular cases in Section \ref{ssec:ns2}.

Note that for $\deg(a)<s$, this coincides precisely with the maps of left multiplication in the cyclic semifield $\SSS_f$, as defined in \cite{Petit}.
\[
M_a(b+Rf) := (ab \modr f) +Rf.
\]
Thus 
\[
\C(\SSS_f)\simeq \{M_a:\deg(a)<s\} \simeq \{a+RF(x^n):\deg(a)<s\}.
\]
\begin{remark}
In \cite{LaShSkew} it was shown that if $f$ and $g$ have the same minimal central left multiple, then $\SSS_f$ and $\SSS_g$ are isotopic. In this new formulation, this is immediately clear; the choice of different divisors of $F(x^n)$ give the action of the same set of elements on isomorphic vector spaces, and hence the spread sets are equivalent.
\end{remark}

\subsection{Minimal central left multiple}

In this section we present an explicit method for computing the minimal central left multiple of an element of $R$, and prove certain useful relations between the coefficients of an element $f$ and its minimal central left multiple. This will allow us to generalise Petit's construction.

For an element $f\in R$, define a semilinear map $\phi_f$ on $L^s$ as follows. We identify the tuple $(v_0,\ldots,v_{k-1})$ in $L^s$ with the polynomial $\sum_{i=0}^{k-1} v_i x^i$, and define 
\[
\phi_f(v) := xv \modr f.
\]
Note that this is the same definition as the map $L_x$ defined above, but we use different notation in order to explicitly distinguish between maps on $E^n$ and maps on $L^s$. It is clear that $\phi_f$ is indeed semilinear, as 
\[
\phi_f(av) = x(av) \modr f = a^{\sigma}(xv \modr f )= a^{\sigma}\phi_f(v)
\] 
for all $a \in L$. By choosing the canonical basis $\{1,x,\ldots,x^{k-1}\}$ for $L^s$, and writing the elements of $L^s$ as column vectors, we see that $\phi_f = C_f\circ \sigma$, where $\sigma$ acts entry-wise on vectors, and $C_f$ denotes the companion matrix of the polynomial $f$;
\[
\phi_f(v) = C_f\cdot v^{\sigma}=\npmatrix{0&0&\cdots&0&-f_0\\1&0&\cdots&0&-f_1\\\vdots&\ddots&\cdots&\cdots&\vdots\\0&0&\cdots&1&-f_{s-1}}\npmatrix{v_0^{\sigma}\\v_1^\sigma\\\vdots\\v_{s-1}^{\sigma}}.
\]
Now the map $\phi_f^n$ defines a linear map on $L^s$; it corresponds to $v\mapsto x^n v \modr f$, and we will denote its matrix with respect to the same basis by $A_f$. Then we have that
\[
A_f = C_f C_f^{\sigma}\cdots C_f^{\sigma^{n-1}},
\]
where again $\sigma$ acts entry-wise on matrices. In \cite{LeBorgne}, the characteristic polynomial of $A_f$ was referred to as the {\it semi-characteristic polynomial} of the semilinear transformation $\phi_f$.

\begin{theorem}
\label{thm:minpoly}
The minimal central left multiple of a monic element $f\in R$ of degree $s$ is equal to the minimal polynomial of the matrix $A_f$ over $K$.
\end{theorem}

\begin{proof}
Let $F$ denote the minimal central left multiple of $f$, and let $G$ denote the minimal polynomial of $A_f$, both of which lie in $K[y]$. Then $G(A_f)\equiv 0$. But $G(A_f)v = G(x^n)v \modr f$ for all $v$ of degree less than $s$, and so choosing $v=1$ gives that $G(x^n)\equiv 0 \modr f$; i.e. $G(x^n)$ is divisible by $f$. Therefore $G$ is divisible by $F$ in $K[y]$.

Conversely, $F(A_f)v = F(x^n)v\modr f = vF(x^n) \modr f\simeq 0 \modr f$ for all $v$, and so $F(A_f)\equiv 0$. Hence $F$ is divisible by  $G$ in $K[y]$, and as they are both monic, they must be equal.
\end{proof}

\begin{example}\label{ex:af}
Suppose $L$ is two-dimensional over $K$, and $f = x^2-\alpha x-\beta \in R$. Then 
\[
C_f = \npmatrix{0&\beta \\1&\alpha };\quad A_f = \npmatrix{0&\beta \\1&\alpha }\npmatrix{0&\beta ^{\sigma}\\1&\alpha ^{\sigma}} = \npmatrix{\beta &\alpha ^{\sigma}\beta \\\alpha &\beta ^{\sigma}+\alpha ^{\sigma+1}}
\]

Then the characteristic polynomial of $A_f$ is $G(y) = y^2 -(\beta +\beta ^{\sigma}+\alpha ^{\sigma+1})y+\beta ^{\sigma+1}$, which is in $K[y]$. A direct calculation shows that 
\[
G(x^2) =x^4 -(\beta +\beta ^{\sigma}+\alpha ^{\sigma+1})x^2+\beta ^{\sigma+1}= (x^2+  \alpha x -\beta ^{\sigma})(x^2-\alpha x-\beta ).
\]

As $f$ is a right-divisor of $G(x^2)$, the minimal central left multiple of $f$ divides $G(x^2)$. Furthermore, $f$ is irreducible if and only $G(y)$ is irreducible in $K[y]$. For example, taking $\alpha =0$ gives us that $x^2-\beta $ is irreducible in $R$ if and only if $y^2-(\beta +\beta ^{\sigma})y+\beta ^{\sigma+1}$ is irreducible in $K[y]$, which occurs if and only if $\beta \notin K$. If $\alpha =0$ and $\beta \in K$, then $f$ is in the centre, and so is equal to its own minimal central left multiple; in this case, $f = x^2-\beta  = (x+\gamma^{\sigma})(x-\gamma)$, where $\gamma\in L$ is any element such that $\gamma^{\sigma+1}=\beta $.

If $\alpha \ne 0$ and $q$ is odd, then $f$ is irreducible if and only if $\Delta = (\beta +\beta ^{\sigma}+\alpha ^{\sigma+1})^2-4\beta ^{\sigma+1}$ is a non-square in $K$.

\end{example}

\subsection{Divisors of $F(x^n)$}\label{ssec:div}

The expression for $A_f$ in terms of $C_f$ gives us the following immediate result, which will be one of the keys to the construction in a later section.
\begin{theorem}
If $f\in L[x;\sigma]$ is monic, irreducible, with $\deg(f)=s$, and $F\in K[y]$ is its minimal central left multiple, then
\[
N(f_0) = (-1)^{s(n-1)} F_0,
\]
where $N(a) = a^{1+\sigma+\cdots+\sigma^{n-1}}$ is the norm map from $L$ to $K$, and $f_0$ and $F_0$ are the constant coefficients of $f$ and $F$ respectively.
\end{theorem}

\begin{proof}
As $f$ is irreducible, so is $F$. As $F$ is equal to the minimal polynomial of $A_f$ by Theorem \ref{thm:minpoly}, we have that $F$ is equal to the characteristic polynomial of $A_f$. Therefore $F(y) = \det(yI-A_f)$, and so $F_0=F(0) = \det(-A_f) = (-1)^s\det(A_f)$.

As $A_f = C_f C_f^{\sigma}\cdots C_f^{\sigma^{n-1}}$, we have that $\det(A_f) = N(\det(C_f))$. Now $\det(C_f) = (-1)^s f_0$, and so $F_0=(-1)^s\det(A_f) = (-1)^s N((-1)^s f_0) = (-1)^{ns+s}N(f_0)$ , proving the claim.
\end{proof}

In Example \ref{ex:af} above, where $n=s=2$, we saw that $f_0=-\beta$, $N(f_0) = \beta^{\sigma+1}$, and $F_0 = \beta^{\sigma+1}$.

\begin{theorem}
If $F$ is a monic irreducible of degree $s$ in $K[y]$, and $g$ is a monic divisor of $F(x^n)$ in $L[x;\sigma]$ of degree $sk$, then 
\[
N(g_0) = (-1)^{ks(n-1)} F_0^k.
\]
\end{theorem}
\begin{proof}
We know that $g$ factorises into monic irreducibles of degree $s$. Let $g=h_1h_2\cdots h_k$ be such a factorisation. By Theorem \ref{thm:norm}, we have that the norm of the constant coefficient of $h_i$ is $(-1)^{s(n-1)} F_0$ for each $i$, and so since $g_0$ is the product of the constant coefficients of the $h_i$'s, the claim follows immediately.
\end{proof}

\begin{remark}
Choosing $s=1$ and $F = y-1$ returns the results of \cite{SheekeyMRD}.
\end{remark}

\subsection{The rank of a skew polynomial in $R_F$}\label{ssec:rank}

Let us fix $F$ a monic irreducible of degree $s$ in $K[y]$, and let $E_F = Z(R_F) \simeq \frac{K[y]}{K[y]F}$. By abuse of notation, let us identify an element $a$ of $R$ with the image of $a+RF(x^n)$ in $M_n(E_F)$. The following proposition demonstrates how to calculate the rank of a polynomial in the matrix ring defined by $F$.

\begin{proposition}
\label{prop:rank}
The rank of the polynomial $a$ as an element of $M_n(E_F)$ is given by
\[
\rank(a) = n - \frac{1}{s}\deg(\gcrd(a,F(x^n))).
\]
\end{proposition}

\begin{proof}
Define $\mathrm{Ann_r}(A) = \langle B \in M_n(E_F) : AB=0\rangle$; that is, the subspace of elements of $M_n(E_F)$ annihilated by left multiplication by $A$.

Note that the rank of $A\in M_n(E_F)$ satisfies
\[
n \cdot (n-\rank(A)) = \dim_{E_F}(\Ann_r(A)).
\]
Let $\gamma = \gcrd(a,F)$, $F= \delta\gamma$. Then as $R$ is a domain, we also have $F= \gamma\delta$. Let $b$ be the unique element of $R$ such that $a = b\gamma$; it follows that $\gcrd(b,\delta)=1$.

Let $v = \delta u +w$, where $\deg(w)< \deg(\delta)$ and $\gcld(w,\delta)=1$. Such elements are unique and can be calculated from the left Euclidean algorithm. Then $av = a\delta u +aw = b\gamma\delta u +b\gamma w = bFu+b\gamma w$, and so
\[
av \modr F \equiv b\gamma w \modr F.
\]
It suffices to show that this is zero if and only if $w=0$, whence $\Ann_r(a) = \{\delta u: \deg(u) < \deg(\gamma)\}$, which clearly has dimension $n\deg(\gamma)$ over $K$, and so dimension $\frac{n\deg(\gamma)}{s}$ over $E_F$. proving the claim.

Now as $\gcrd(b,\delta)=1$ there exist $c,d \in R$ such that $cb+d\delta=1$. Then $cb\gamma+d\delta\gamma=\gamma$, and so $cb\gamma \equiv \gamma \modr F$.

Now $b\gamma w\equiv 0 \modr F$ implies that $cb\gamma w \equiv 0 \modr F$. But since $cb\gamma w \equiv \gamma w \modr F$ (using the fact that $F$ is in the centre), we get $\gamma w \equiv 0 \modr F$. But $\deg(w)< \deg(\delta)$, and so $\deg(\gamma w)<\deg(F)$. Hence this can occur if and only if $w=0$, completing the proof.
\end{proof}

\begin{remark}
Taking $L=\Fqn$, $s=1$, $F(y) = y-1$, $x^\sigma = x^q$, we can see that this matches with the rank of a corresponding linearized polynomial as a linear map on $L$. We identify $a=\sum_{i=0}^{n-1}a_i x^i$ with $A = \sum_{i=0}^{n-1}a_i X^{q^i}$. Then as the number of roots of $A$ in $L$ is $\deg(\gcd(A,X^{q^n}-X))$, we have
\[
\rank(A) = n-\log_q \deg(\gcd(A,X^{q^n}-X)).
\]
Now since it is known \cite{Lidl} that $\deg(\gcd(A,B)) = q^{\deg(gcrd(a,b))}$ for any skew polynomial $b$ with corresponding linearised polynomial $B$, then we see that
\[
\rank(a)=n - \deg(\gcrd(a,x^n-1)) = \rank(A).
\]
\end{remark}

We can combine Proposition \ref{prop:rank} with Theorem \ref{thm:norm} to get the following.

\begin{theorem}
\label{thm:norm}
If $a\in R$ is a polynomial of degree at most $sk$, then the rank of $a$ as an element of $M_n(E)$ is at least $n-k$. Furthermore, if the rank of $a$ is equal to $n-k$ then $\frac{N(a_0)}{N(a_{sk})} = (-1)^{sk(n-1)} F_0^k.$
\end{theorem}

\begin{proof}
If $\deg(a)\leq sk$ and $\rank(a) = n-k$, we must have that $\deg(a)= sk$ and $a$ is a divisor of $F(x^n)$. Then $a_{sk}^{-1}a$ is a monic divisor of $F(x^n)$, and so by Theorem \ref{thm:norm}, we have
\[
N(a_0/a_{sk}) = (-1)^{ks(n-1)} F_0^k,
\]
and since the norm map is multiplicative, the claim is proven.
\end{proof}

%
%
%
%
%

\section{New Constructions}
\label{sec:const}

We are now ready to introduce our new family of MRD codes.

\begin{theorem}
\label{thm:construction}
Let $L$ be a field, $\sigma$ an automorphism of $L$ with fixed field $K$, and $\rho$ an automorphism of $L$ over some field $K'\leq K$. Let $R= L[x;\sigma]$, let $F$ be an irreducible polynomial in $K[y]$ of degree $s$, $E=\frac{K[y]}{(F(y))}$, and $R_F = \frac{R}{RF(x^n)}$.

Then the set
\[
S_{n,s,k}(\eta,\rho,F) := \{a+ RF(x^n):\deg(a) \leq ks, a_{ks} = \eta a_0^{\rho}\}
\]
defines a $K'$-linear MRD code in $R_F\simeq M_n(E_F)$ with minimum distance $n-k+1$ for any $\eta\in L$ such that $N_{L:K'}(\eta)N_{K:K'}((-1)^{sk(n-1)}F_0^k) \ne 1$.
\end{theorem}

\begin{proof}
Clearly $S_{n,s,k}(\eta,\rho,F)$ has dimension $nks[K:K']$ over $K'$, while $R_F$ has dimension $n^2s[K:K']$ over $K'$. Thus it suffices to show that the minimum rank of a non-zero element of $S_{n,s,k}(\eta,\rho,F)$ is $n-k+1$.

If $\eta=0$ or $a_0=0$, then this follows immediately from Proposition \ref{prop:rank}, since $\deg(\gcrd(a,F(x^n)))\leq s(k-1)$.

If $a_0\eta\ne 0$, then $a\in S_{n,s,k}(\eta,\rho,F)$ has rank at least $n-k$, and by Theorem \ref{thm:norm} has rank equal to  $n-k$ only if $N_{L:K}(a_0/a_{sk}) = N_{L:K}(a_0)^{1-\rho}N(\eta^{-1}) = (-1)^{ks(n-1)} F_0^k$. Taking the norm from $K$ to $K'$ of both sides completes the proof.
\end{proof}
When $n$, $s$, and $F$ are clear we will write $S_{k}(\eta,\rho)$. When $L=\Fqn$, $S_{n,s,k}(\eta,\rho,F)$ is an MRD code in $M_n(\Fqs)$ of size $q^{nsk}$ with minimum distance $n-k+1$.

\begin{remark}
Setting $k=1$ gives a family of semifields. Setting $\eta=0$, $k=1$ returns the cyclic semifields. Setting $k=1$, $s=1$ returns the Generalised Twisted Fields. Hence we have provided a construction which incorporates two of the most general known constructions into a single family. We will show in the next section that this family also contains new semifields.
\end{remark}

\begin{remark}
Setting $\eta=0$, $s=1$, $F=y-1$ returns the generalised Gabidulin codes; when $s=1$ we take $f=x-1$, $F(x^n)=x^n-1$, and we have the usual correspondence between skew polynomials $\sum_{i=0}^{k-1} a_i x^i$ and linearized polynomials $\sum_{i=0}^{k-1} a_i X^{\sigma^i}$.

Setting $s=1$, $F=y-1$, $\eta\ne 0$ returns the twisted Gabidulin codes $H_k(\eta,h)$ when $a^{\rho} = a^{\sigma^h}$, as well as the generalisations from \cite{Ozbudak1} when $\rho$ is an automorphism of $L$ which does not fix $K$.
\end{remark}


\subsection{Worked example; $n=s=2$}
\label{ssec:ns2}

Here we will give an explicit correspondence between the quotient ring $R/RF(x^n)$ and the matrix algebra $M_n(\Fqs)$ for the case $n=s=2$, and hence give a representation of all spread sets arising from this new construction in this special case. We note that all such semifields have been classified in \cite{CAPOTR2004}, where it was shown that for $q$ odd, all examples are equivalent to one of four known constructions; generalised twisted field, generalised Dickson semifield, Hughes-Kleinfeld semifield, and semifields of Cordero-Figueroa type. All of these can be realised as examples from our new construction.

Let us take $F\in \Fq[y]$ irreducible of degree $2$, and let $\beta $ be a root of $F$ in $\FF_{q^2}$. Then $x^2-\beta$ is irreducible and divides $F(x^2)$ in $\FF_{q^2}[x;\sigma]$, since 
\[
(x^2-\beta ^{\sigma})(x^2-\beta ) = x^4-(\beta +\beta ^{\sigma})x^2 +\beta ^{\sigma+1} = F(x^2).
\]
Let $\beta+\beta^{\sigma} = \lambda,$ $\beta^{\sigma+1} = \mu$, and so $F(y) = y^2-\lambda y+\mu$, and $\beta^2 = \lambda \beta - \mu$.

Then as $x^2+Rf = \beta +Rf$, we have that $E_f = \langle 1+Rf,\beta+Rf\rangle_{\Fq} = \{\alpha +Rf:\alpha \in \Fqt\}\simeq \Fqt$. In the following we will write $\alpha$ for $\alpha +Rf$ for brevity. It is clear that $\{1+Rf,x+Rf\}$ is an $E_f$-basis for $V_f$.

For any $a_0\in \Fqt$, we have that
\begin{align*}
M_{a_0}(1+Rf) &= a_0 +Rf = (1+Rf)(a_0+Rf),\\
M_{a_0}(x+Rf) &= a_0 x+Rf = (x+Rf)(a_0^{\sigma}+Rf).
\end{align*}

Hence with respect to this basis, and writing elements of $V_f$ as column vectors with entries in $E_f$, we get that
 \[
 M_{a_0} = \npmatrix{a_0&0\\0&a_0^{\sigma}}.
 \]
Next, we have
\begin{align*}
M_{a_1 x}(1+Rf) &= a_1 x+Rf = (x+Rf)(a_1 ^{\sigma}+Rf),\\
M_{a_1 x}(x+Rf) &= a_1 x^2+Rf = (1+Rf)(a_1 \beta+Rf)
\end{align*}
and so 
\[
M_{a_1 x} = \npmatrix{0&a_1 \beta\\a_1^{\sigma}&0}.
\]
Similarly we get that 
\[
M_{a_2 x^2} = \npmatrix{a_2\beta&0\\0& a_2^{\sigma}\beta},\quad\quad
M_{a_3 x^3} = \npmatrix{0&a_3 \beta^2\\a_3^{\sigma}\beta&0}.
\]
 
Hence we get
 \[
 M_{a_0+a_1x+a_2x^2+a_3x^3} = \npmatrix{a_0+a_2\beta&a_1\beta+a_3\beta^2\\a_1^{\sigma}+a_3^{\sigma}\beta&a_0^{\sigma}+a_2^{\sigma}\beta}.
 \]

Then the spread set of the presemifield $S_1(\eta,\rho)$ can be represented as
\[
\left\{\npmatrix{a_0+\eta a_0^{\rho}\beta&a_1\beta\\a_1^{\sigma}&a_0^{\sigma}+\eta^\rho a_0^{\rho\sigma}\beta}:a_0,a_1\in \Fqt\right\}.
\]
Choosing $\eta=0$ returns the cyclic semifields, with spread set
\[
\left\{\npmatrix{a_0&a_1\beta\\a_1^{\sigma}&a_0^{\sigma}}:a_0,a_1\in \Fqt\right\}.
\]
In this case these semifields coincide with the Hughes-Kleinfeld construction. It is clear that all matrices in this set are invertible: the determinant of an element is $a_0^{\sigma+1}-a_1^{\sigma+1}\beta$, and since $a_i^{\sigma+1}\in \Fq$, and $\beta\notin \Fq$, this can be zero if and only if $a_0=a_1=0$.

Choosing $\rho=\sigma$, we get 
\[
\left\{\npmatrix{a_0+\eta \beta a_0^{\sigma}&a_1\beta\\a_1^{\sigma}&a_0^{\sigma}+\eta^\rho a_0^{\rho\sigma}\beta}:a_0,a_1\in \Fqt\right\}.
\]
Choosing $\rho=\sigma$, $\eta^{\sigma+1}\beta^2=1$, this is equivalent to 
\[
\left\{\npmatrix{a_0&a_1\\ \eta^{-\sigma} a_1&a_0}:a_0,a_1\in \Fqt\right\},
\]
which is the spread set of a generalised Dickson semifield.

Choosing $\rho=1$ gives
\[
\left\{\npmatrix{a_0+\eta\beta a_0^\sigma&a_1\beta \\  a_1^\sigma&(1+\eta\beta)a_0^\sigma}:a_0,a_1\in \Fqt\right\},
\]
which is the spread set of a semifield of Cordero-Figueroa type.


Hence we have shown the following:
\begin{theorem}
All semifields of order $q^4$ with centre of order $q$ and right nucleus of order $q^2$ are isotopic to a semifield of the form $S_{n,s,1}(\eta,\rho)$, for some $\eta\in \Fqn$, $\rho \in \Aut(\Fqn)$, $[n,s]\in \{[2,2],[4,1]\}$.
\end{theorem}

%

\subsection{Worked example; $n=s=3$}
\label{ssec:explicit}

Let $F(y) = (y-\gamma)(y-\gamma^{\sigma})(y-\gamma^{\sigma^2})$ for $\gamma \in \FF_{q^3}\backslash\Fq$. Then $f=x^3-\gamma$ is an irreducible divisor of $F(x^3)$. 

\[
E_F = \langle 1,x^3,x^6\rangle = \langle 1,\gamma,\gamma^2\rangle.
\]
\[
V_f = \langle 1,x,x^2\rangle_{E_F}.
\]
Then
\[
M_a = \npmatrix{a&0&0\\0&a^{\sigma^2}&0\\0&0&a^{\sigma}};
\]
\[
M_x = \npmatrix{0&0&\gamma\\1&0&0\\0&1&0};
\]

\[
M_g = \npmatrix{g_0+g_3\gamma+g_6\gamma^2&g_2\gamma+g_5\gamma^2+g_8\gamma^3&g_1\gamma+g_4\gamma^2+g_7\gamma^3\\g_1^{\sigma^2}+g_4^{\sigma^2}\gamma+g_7^{\sigma^2}\gamma^2&g_0^{\sigma^2}+g_3^{\sigma^2}\gamma+g_6^{\sigma^2}\gamma^2&g_2^{\sigma^2}\gamma+g_5^{\sigma^2}\gamma^2+g_8^{\sigma^2}\gamma^3\\g_2^{\sigma}+g_5^{\sigma}\gamma+g_8^{\sigma}\gamma^2&g_1^{\sigma}+g_4^{\sigma}\gamma+g_7^{\sigma}\gamma^2&g_0^{\sigma}+g_3^{\sigma}\gamma+g_6^{\sigma}\gamma^2};
\]
Hence the spread set of $S_{3,3,1}(\eta,\rho)$ is
\[
\C(S_{3,3,1}(\eta,\rho)) = \left\{\npmatrix{g_0+\eta g_0^\rho\gamma&g_2\gamma&g_1\gamma\\g_1^{\sigma^2}&g_0^{\sigma^2}+(\eta g_0^\rho)^{\sigma^2}\gamma&g_2^{\sigma^2}\gamma\\g_2^{\sigma}&g_1^{\sigma}&g_0^{\sigma}+(\eta g_0^\rho)^{\sigma}\gamma} : g_i \in \FF_{q^3}\right\}.
\]
We have that $S_{3,3,1}(\eta,\rho)$ is a semifield of order $q^9$, with right nucleus of order $q^3$. If $\rho=\sigma$ then its left and middle nuclei and centre all have order $q$. As far as the author is aware, there are no known constructions for semifields with these parameters.

\section{Proof that the family contains new codes}
\label{sec:isnew}

\subsection{Nuclei of semifields in the family $S_{n,s,1}$}

Due to the range of existing constructions for semifields, it is difficult to ascertain in general when this construction leads to new semifields. In this paper we will be satisfied with showing that there are certain parameters where this construction is definitely new. We do this by calculating the nuclei. More generally for MRD codes it is much easier to guarantee newness, as there are fewer existing constructions to consider.

\begin{theorem}\label{thm:nuclei}
Let us assume that $k\leq \frac{n}{2}$, and $sk>1$. Let $\C=S_{n,s,k}(\eta,\rho,F)$, and let $\C'$ be any code equivalent to $\C$ containing the identity.

If $\eta\ne 0$ then
\[
\I_l(\C') \simeq L_{\rho}
\]
\[
\I_r(\C') \simeq L_{\sigma^{sk}\rho^{-1}}
\]
\[
C(\C') \simeq Z(R_F)
\]
\[
Z(\C') \simeq K_\rho.
\]

If $\eta=0$ then $S_{n,s,k}(0,\rho,F)=S_{n,s,k}(0,0,F)$ for all $\rho$, and
\[
\I_l(\C') \simeq L
\]
\[
\I_r(\C') \simeq L
\]
\[
C(\C') \simeq Z(R_F)
\]
\[
Z(\C') \simeq K.
\]

\end{theorem}

\begin{proof}
Let $\C= S_{n,s,k}(\eta,\rho,F)$. First we compute $\I_\ell(\C)$. We claim that $\I_\ell(\C)$ consists of constant polynomials. 

Suppose $g\in \I_\ell(\C)$, and $\deg(g)\leq sk$. Then since $sk>1$, we have $x^{sk-1}\in \C$, and hence $gx^{sk-1}\mod F(x^n)\in \C$. Now since $\deg(gx^{sk-1})\leq 2sk-1<ns$, we have $gx^{sk-1} \in \C$. But $gx^{sk-1} = g_0x^{sk-1}+g_1x^{sk}+\ldots g_{sk}x^{2sk-1}$, and so $g = g_0$. Clearly if $a\in \C$ with $a_0=0$, then $g_0 a\in \C$ for all $g_0\in L$. Now $g_0(a_0+\eta a_0^{\rho}x^s)\in \C$ if and only if $g_0\eta a_0^{\rho} = \eta (g_0a_0)^\rho$, and so if $\eta\ne0$ then $g_0^{\rho}=g_0$. Hence 
\[
\{g\in \I_\ell(\C)|\deg(g)\leq sk\} \simeq \left\{\begin{array}{ll}L_{\rho}&\eta \ne 0\\L&\eta=0.\end{array}\right.
\]
A similar argument shows that if $g\in\I_r(\C)$ and $\deg(g)\leq sk$, then $g=g_0$. Then $(a_0+\eta a_0^{\rho}x^s)g_0\in \C$ if and only if $g_0^{\sigma^{sk}}\eta a_0^{\rho} = \eta (g_0a_0)^\rho$, and so if $\eta\ne 0$ then $g_0^{\rho}=g_0^{\sigma^{sk}}$. Hence 
\[
\{g\in \I_r(\C)|\deg(g)\leq sk\} \simeq \left\{\begin{array}{ll}L_{\sigma^{sk}\rho^{-1}}&\eta \ne 0\\L&\eta=0.\end{array}\right.
\]
Hence to complete the calculation of the left and right idealisers, it suffices to show that an element of either idealiser has degree at most $sk$. If $\eta=0$, then $1\in \C$, and so $\I_{\ell}(\C),\I_r(\C)\subset \C$. Hence we assume for the remainder of this proof that $\eta \ne 0$.

If $g\in \I_\ell(\C)$, we must have that $gax^m\mod F(x^n)\in\C$ for all $a\in L$, $m\in \{1,\ldots,x^{sk-1}\}$. As $sk>1$, this set is non-empty. 

Consider $m=1$. Then 
\[
gx \mod F(x^n) = \left(\sum_{i=1}^{ns-1} g_{i-1}x^i\right) - g_{ns-1}\left( \sum_{j=0}^{s-1}F_j x^{nj} \right).
\]
Hence for all $i\in \{sk+1,\ldots,ns-1\}$ we have
\[
g_{i-1} = \left\{\begin{array}{ll} 0 &i \not\equiv 0 \mod n\\g_{ns-1}F_{i/n}&i\equiv 0 \mod n. \end{array}\right.
\]
 We want to show that $g_{ns-1}=0$, implying $\deg(g) \leq ks-1$. The coefficient of $x^{ks}$ in $[gx\mod F(x^n)]$ is $g_{ks-1}-g_{ns-1}F_{ks/n}$, while the constant coefficient is $-g_{ns-1}F_0$. Therefore as $[gx\mod F(x^n)] \in \C$, we need
\[
g_{ks-1} -g_{ns-1}F_{ks/n}= -\eta F_0^{\rho}g_{ns-1}^\rho.
\]
Now if $ks>2$, we also have that $[gx^2 \mod F(x^n)]\in \C$. The coefficient of $x^{sk+1}$ in $[gx^2 \mod F(x^n)]$ is $g_{ks-1} -g_{ns-1}F_{ks/n}$, and so this must be zero. Hence we have that $-\eta F_0^{\rho}g_{ns-1}^\rho=0$, and since $\eta\ne 0 $ and $F_0\ne 0$, we must have $g_{ns-1}=0$, as claimed.

Suppose now $ks=2$. If $s=1$, then we are in the case of twisted gabidulin codes, for which the idealisers were calculated in \cite{SheekeyMRD}, \cite{LuTrZh2015}. If $[k,s]=[1,2]$, then $g = g_0+g_1x + g_{2n-1} (x^{2n-1}+F_1 x^{n-1})$. Now $g(1+\eta x^2) \mod F(x^n)\in \C$, and since the coefficient of $x^{2n-1}$ in this is $g_{2n-1}$, we have $g_{2n-1}=0$ and $\deg(g)<2=sk$.

Hence the left idealiser is as claimed. The proof for the right idealisers is similar. In order to use Proposition \ref{prop:nuc}, we need to consider a set equivalent to $\C$ which contains the identity, and calculate its centraliser and centre.  

Let $z(x^n)\in Z(R)$ be such that $\deg(z(x^n))<sn$ and $z(x^n)x^{ns} \equiv 1 \mod F(x^n)$. Such a $z$ exists, since $Z(R)/Z(R)F(x^n)$ is a field. Then $z(x^n)x^{ns-1}$ is the inverse of $x$ in $R_F$. Hence
\[
\C' := \C z(x^n)x^{ns-1}
\]
is a spread set equivalent to $\C$, and $\C'$ contains the identity. We have that
\[
\C' = \{ a_1+a_2x+a_3x^2+\cdots + a_{sk-2}x^{sk-2} +a_0z(x^n)x^{ns-1}+\eta a_0^\rho x^{s-1} :a_i \in L\}.
\]
Now as $a_1\in \C$ for all $a_1\in L$, we have that if $g\in C(\C')$ then $ga_1-a_1g \in RF(x^n)$ for all $a_1\in L$. As $\deg(ga_1-a_1g)< ns$, we must have $g\in L[x^n;\sigma]$, and $\deg(g)\leq n(s-1)$. Furthermore, as $x\in \C'$, we have that if $g\in C(\C')$ then $gx-xg \in RF(x^n)$. But as $\deg(g)\leq n(s-1)$, we must have $gx=xg$, and so $g \in K[x;\sigma]$. Therefore we must have $g\in K[x^n;\sigma] = Z(R)$. Hence $C(\C') = E_F$.

Now by Proposition \ref{prop:idequiv}, $\I_\ell(\C') = \I_\ell(\C)$. Hence $Z(\C') = \I_\ell(\C')\cap C(\C') = \{g_0:g_0\in K,g_0^\rho=g_0\}\simeq K_\rho$, completing the proof.
\end{proof}

\begin{corollary}
\label{cor:nuc}
The finite semifields $S_{n,s,1}(\eta,\rho,F)\leq M_{n}(\FF_{q^s})\leq M_{ns}(\Fq)$ with $q = p^{e}$, $x^{\rho}=x^{p^i}$, $x^{\sigma}= x^{p^j}$, $s>1$, have parameters 
\[
(p^{nse},p^{(ne,i)},p^{(ne,sj-i)},p^{s(ne,j)},p^{(ne,i,j)})
\]
when $s>1$, while $S_{n,s,1}(0,0,F)$ has parameters 
\[
(p^{nse},p^{ne},p^{ne},p^{se},p^e).
\]
\end{corollary}

\begin{remark}
When $s=k=1$, we are in the case of a Generalised Twisted Field. The nuclei of these were calculated by Albert \cite{Albert1961}: a semifield with multiplication
\[
xy-\eta x^{p^i}y^{p^j},\quad x,y\in \FF_{p^{ne}}
\]
has nuclear parameters
\[
(p^{ne},p^{(ne,i)},p^{(ne,j-i)},p^{(ne,j)},p^{(ne,i,j)}).
\]
We can see that this is precisely the parameters found by setting $s=1$ in Corollary \ref{cor:nuc}.
\end{remark}


Note that this also implies the results in \cite{TrZh} regarding the parameters of the the twisted Gabidulin codes.

\subsection{Nuclei of known constructions for semifields}

There are many known constructions for semifields. However, most of them are only valid for certain characteristics or nuclei. For example, many constructions focus on the case of semifields two-dimensional over a nucleus. Other constructions which work for many dimensions over a nucleus (for example, that of Pott-Zhou \cite{PottZhou}), restrict to the case of commutative semifields. In a commutative semifield we have that $|Z| = |\NN_l|=|\NN_r| \leq |\NN_m|$. If we avoid these cases, as well as the parameters of any generalised twisted field, then we can be sure that this construction gives new semifields.

For example, to the author's knowledge there is no known construction for a semifield of order $q^{12}$ with centre of order $q$ and each nucleus having order $q^2$. The semifield $S_{6,2,1}(\eta,\rho)$ with $x^{\rho} = x^{q^4}$ has these properties, and hence is new.

%

\begin{theorem}
The family $S_{n,s,1}(\eta,\rho,F)$ contains new semifields for some choices of $n,s$.
\end{theorem}

Other parameters for which there were no known constructions, and for which the new construction gives semifields with these parameters, are tabulated here. Here we take $L = \FF_{q^{ne}}$, $x^\sigma = x^{q^e}$, $x^{\rho}=x^{q^i}$. We arbitrarily restrict to semifields at most $20$-dimensional over their centre.
\[
\begin{array}{c|c}
(n,s,e,i)&N\\
\hline
(3,3,1,3)&(q^{9},q^3,q^3,q^3,q)\\
(6,2,1,4)&(q^{12},q^2,q^2,q^2,q)\\
(8,2,1,4)&(q^{16},q^4,q^2,q^2,q)\\
(6,3,1,0)&(q^{18},q^6,q^3,q^3,q)\\
\end{array}
\]
It is likely that this construction produces new semifields for almost all parameters. However it requires further research to establish this precisely. Furthermore, different choices for the irreducible polynomial $F$ can lead to non-isotopic semifields. This again is a topic for further research.

As there are fewer constructions for MRD codes, it is much easier to establish newness, as for almost all parameters the only constructions to compare to are the Twisted Gabidulin codes.

\begin{theorem}
The family $S_{n,s,k}(\eta,\rho,F)$ contains new MRD codes for $k>1$, for almost all $n,s>1$.
\end{theorem}


%
%
%

\end{document}